\title{Improved theoretical guarantees regarding a class of two-row cutting planes}
\author{
Yogesh Awate \\
Tepper School of Business\\
Carnegie Mellon University\\
Pittsburgh, PA 15213
}
\date{January 18, 2012}
\documentclass[11pt]{article}
\usepackage{graphicx}   
\usepackage {amsmath,amssymb,graphicx,fullpage,epsfig,formatfig,url}
\usepackage{amsmath, amsthm, amssymb}
\usepackage{amsmath, fullpage}
\usepackage{amsfonts}
\usepackage{amssymb}
\usepackage{graphicx}
\newtheorem{theorem}{Theorem}[section]
\newtheorem{lemma}[theorem]{Lemma}

\newtheorem{corollary}[theorem]{Corollary}

\begin{document}
\maketitle
\begin{abstract}
The corner polyhedron is described by minimal valid inequalities from maximal lattice-free convex sets. For the Relaxed Corner Polyhedron (RCP) with two free integer variables and any number of non-negative continuous variables, it is known that such facet defining inequalities arise from maximal lattice-free splits, triangles and quadrilaterals. We improve on the tightest known upper bound for the approximation of the RCP, purely by minimal valid inequalities from maximal lattice-free quadrilaterals, from 2 to 1.71. We generalize the tightest known lower bound of 1.125 for the approximation of  the RCP, purely by minimal valid inequalities from maximal lattice-free triangles, to an infinite subclass of quadrilaterals. 
\end{abstract}
\section{Introduction}

\begin{figure}
\begin{center}
 \includegraphics[height=6cm]{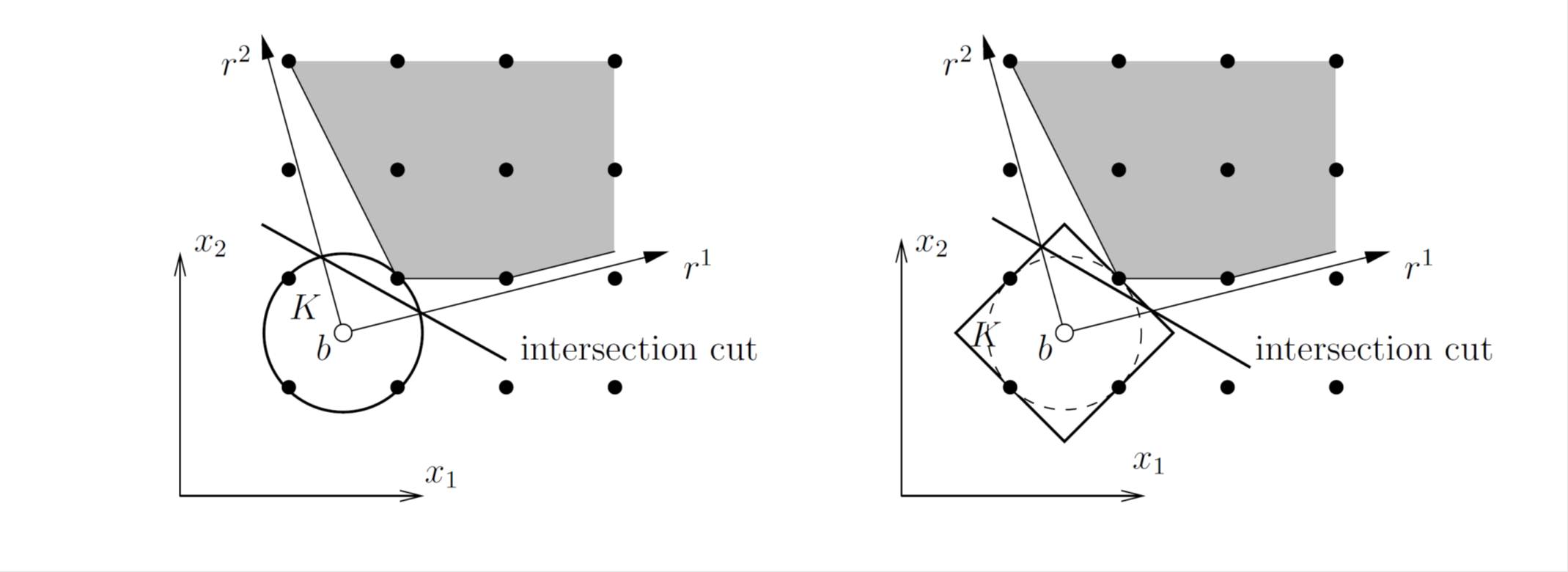}
\end{center}
\caption{Correspondence between intersection cuts and facets of the corner polyhedron. (Cornu{\'e}jols \emph{et al.}~\cite{Ipbible}).}
\label{IC_CP}
\end{figure}

\paragraph{}
Corner polyhedra, introduced by Gomory~\cite{Gomory69}, are obtained by relaxing the non-negativity restrictions on the set of basic variables of a mixed-integer linear program (MILP). This means we can drop the constraints for all continuous basic variables. Any inequality valid for the corner polyhedron is valid for the original polyhedron. Moreover, the set of nontrivial facets of the corner polyhedron is the same (see Figure \ref{IC_CP} for intuition) as the set of undominated intersection cuts  (Cornu{\'e}jols \emph{et al.}~\cite{Ipbible}). The recent paper by Andersen \emph{et al.}~\cite{AndersenLWW07} has led to a renewed interest in corner polyhedra as a tool to generate stronger multi-row cuts as opposed to traditional single-row cuts. The interested reader may refer to Cornu{\'e}jols \emph{et al.}~\cite{Ipbible} for a comprehensive review of corner polyhedra.

We consider the problem in  Andersen \emph{et al.}~\cite{AndersenLWW07} , which has two integer constrained basic variables and any number of nonbasic nonzero continuous variables i.e. MILP: $ x = f + \sum_{j = 1}^{ k} r_js_j, f \in \mathbb{Q}^2 \setminus \mathbb{Z}^2, k \geq 1, x \in \mathbb{Z}^2 , s \in \mathbb{R}_+^k$. We define
$R_f(r_1,\dots,r_k) = \text{Conv}(s \in \mathbb{R}_+^k: f + \sum_{j = 1}^{ k} r_js_j\in \mathbb{Z}^2)$.  It is important to note that the traditional Gomory corner polyhedron with two free integer basic variables would allow for both continuous and discrete nonbasic variables, whereas the above formulation allows only for continuous nonbasic variables. Hence, we call $R_f(r_1,\dots,r_k)$ the Relaxed Corner Polyhedron (RCP), following the notation of Basu  \emph{et al.}~\cite{DBLP:conf/ipco/BasuCM11}.
It is useful to consider the semi-infinite relaxation given by
$R_f = \text{Conv}(s \in \mathbb{R}_+^{\infty}: f + \sum_{r \in \mathbb{Q}^2} rs_r\in \mathbb{Z}^2)$ because the minimal valid inequalities for $R_f$ are known in terms of the minimal valid functions (a.k.a. gauge functions) of maximal lattice-free convex sets containing $f$ in their interior. 
\vspace{3pt}
Any maximal lattice-free (MLF) convex set $B$ defines a minimal valid function $\psi_B: \mathbb{R}^2 \rightarrow \mathbb{R}_+$ such that
$\psi_B(0) = 0$ and $\psi_B(x - f) = 1$ $ \forall x \in \text{Boundary}(B)$. If $\exists \lambda \geq 0$ s.t. $ f+ \lambda r \in \text{Boundary}(B)$, then $\psi_B(r) = 1/\lambda$ else $\psi_B(r) = 0$.
$R_f = \bigcap_{B} ( \sum_{ r \in \mathbb{Q}^2} \psi_B(r)s_r \geq 1)$. A restriction of these infinite dimensional minimal valid functions to the set of extreme rays of the RCP gives minimal valid functions for the finite dimensional RCP i.e.
$R_f(r_1,\dots,r_k) = \bigcap_{ B} ( \sum_{ j = 1}^{k} \psi_B(r_j)s_j \geq 1)$  (Borozan  and Cornu\'{e}jols~\cite{borozan:minimal}). 
Theorems 1.1, 1.2 and 1.3 encapsulate some key properties of maximal lattice-free convex sets.
\begin{theorem}
(Doignon~\cite{doig73}, Bell~\cite{bell77} and Scarf~\cite{scarf77}): Any full dimensional maximal lattice-free convex set $K \subseteq \mathbb{R}^p$ has at most $2^p$ facets.
\end{theorem}

\begin{theorem}
(Lov{\'a}sz~\cite{lov89}): In the plane, a maximal lattice-free convex set with a nonempty
interior is one of the following:
\newline
i) A split $c \leq ax_1 + bx_2 \leq c + 1$ where $a$, $b$ are coprime, $c$ is an integer,
\newline
ii) A triangle with an integral point in the interior of each of its edges,
\newline
iii) A quadrilateral containing exactly four integral points, with exactly one of them in
the interior of each of its edges.
\label{Lov_MLFCS}
\end{theorem}
Figure 2 shows the three types of maximal lattice-free convex sets.

\begin{figure}
\begin{center}
 \includegraphics[height=5cm]{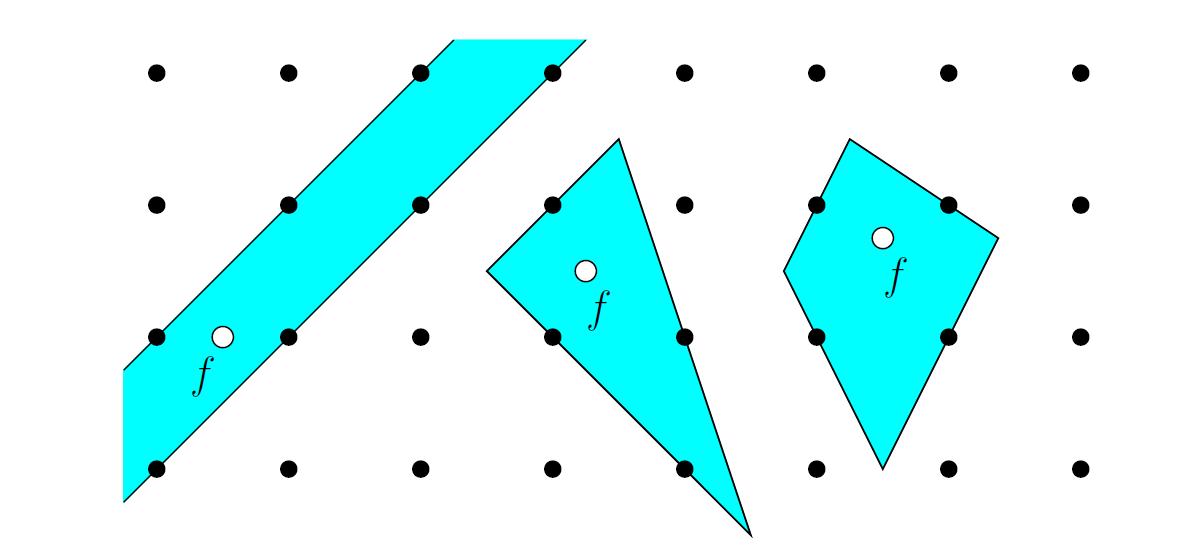}
\end{center}
\caption{Maximal lattice-free convex sets with nonempty interior in $\mathbb{R}^2 \hspace{4pt} $(Cornu{\'e}jols \emph{et al.}~\cite{Ipbible}).}
\end{figure}

\begin{theorem}
(Dey and Wolsey~\cite{Dey08})
The maximal lattice-free triangles are of three types: \newline
(i) Type-1 triangles: triangles with integral vertices and exactly one integral point in the
relative interior of each edge, \newline
(ii) Type-2 triangles: triangles with at least one fractional vertex v, exactly one integral
point in the relative interior of the two edges incident to v and at least two integral
points on the third edge, \newline
(iii) Type-3 triangles: triangles with exactly three integral points on the boundary, one in
the relative interior of each edge.
\end{theorem}
Figure 3 shows the three types of maximal lattice-free triangles.

We define the split closure $S_f (r_1,\dots, r_k)$, triangle closure $T_f (r_1,\dots, r_k)$ and quadrilateral closure $Q_f (r_1,\dots, r_k)$ as follows. 
\begin{equation}
S_f (r_1,\dots, r_k) \hspace{4pt} =  \bigcap_{ B:  B \hspace{4pt}is \hspace{4pt} a \hspace{4pt} maximal \hspace{4pt} lattice-free \hspace{4pt}split} (\hspace{4pt} \sum_{ j = 1}^{k} \psi_B(r_j)s_j \geq 1) .
\end{equation}
\begin{equation}
T_f (r_1,\dots, r_k) \hspace{4pt} = \bigcap_{ B:  B \hspace{4pt}is \hspace{4pt} a \hspace{4pt} maximal \hspace{4pt} lattice-free \hspace{4pt}triangle} (\hspace{4pt} \sum_{ j = 1}^{k} \psi_B(r_j)s_j \geq 1 ).
\end{equation}
\begin{equation}
Q_f (r_1,\dots, r_k) \hspace{4pt} =  \bigcap_{ B:  B \hspace{4pt}is \hspace{4pt} a \hspace{4pt} maximal \hspace{4pt} lattice-free \hspace{4pt}quadrilateral} (\hspace{4pt} \sum_{ j = 1}^{k} \psi_B(r_j)s_j \geq 1) .
\end{equation}

\begin{figure}
\begin{center}
 \includegraphics[height=6cm]{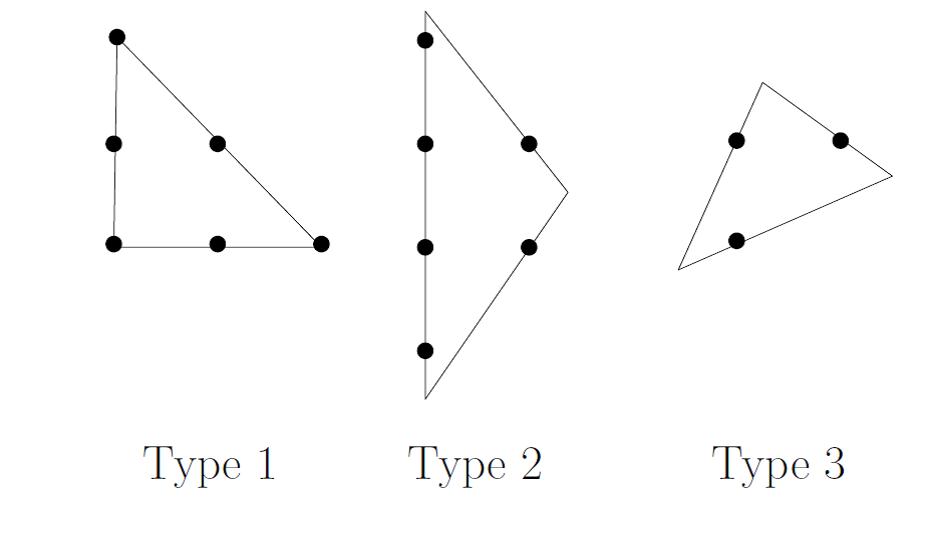}
\caption{Three types of maximal lattice-free triangles (Cornu{\'e}jols \emph{et al.}~\cite{Ipbible}).}
\end{center}
\end{figure}

Theorem \ref{Lov_MLFCS} gives us the following result.

\begin{theorem}
$R_f (r_1,\dots, r_k) = S_f (r_1,\dots, r_k)\cap T_f (r_1,\dots, r_k) \cap Q_f (r_1,\dots, r_k)$.
\vspace{5pt}
\end{theorem}
It is easy to see that maximal lattice-free splits are a limiting case of maximal lattice-free triangles, quadrilaterals. The following theorems follow (see Basu \emph{et al.}~\cite{BasuBCM11} for a rigorous proof).
\begin{theorem}
(Basu \emph{et al.}~\cite{BasuBCM11}): $T_f (r_1,\dots, r_k) \subseteq S_f (r_1,\dots, r_k), Q_f (r_1,\dots, r_k) \subseteq S_f (r_1,\dots, r_k)$.
\end{theorem}
\begin{theorem}
(Basu \emph{et al.}~\cite{BasuBCM11})
$R_f (r_1,\dots, r_k) = T_f (r_1,\dots, r_k) \cap Q_f (r_1,\dots, r_k)$.
\end{theorem}
It is important to note that if we ignore maximality and consider all triangles and quadrilaterals, then the set of triangles  is redundant and the RCP can  be defined just in terms of quadrilaterals. This is because each triangle can be approximated arbitrarily closely by a sequence of quadrilaterals. However, a maximal lattice-free Type-3 triangle cannot be approximated (see Figure \ref{Apprx_T_C}) arbitrarily closely by a sequence of maximal lattice-free quadrilaterals. Hence, considering only maximal sets, we need both maximal lattice-free triangles and quadrilaterals. 

\begin{figure}
\begin{center}
\includegraphics[height=8cm]{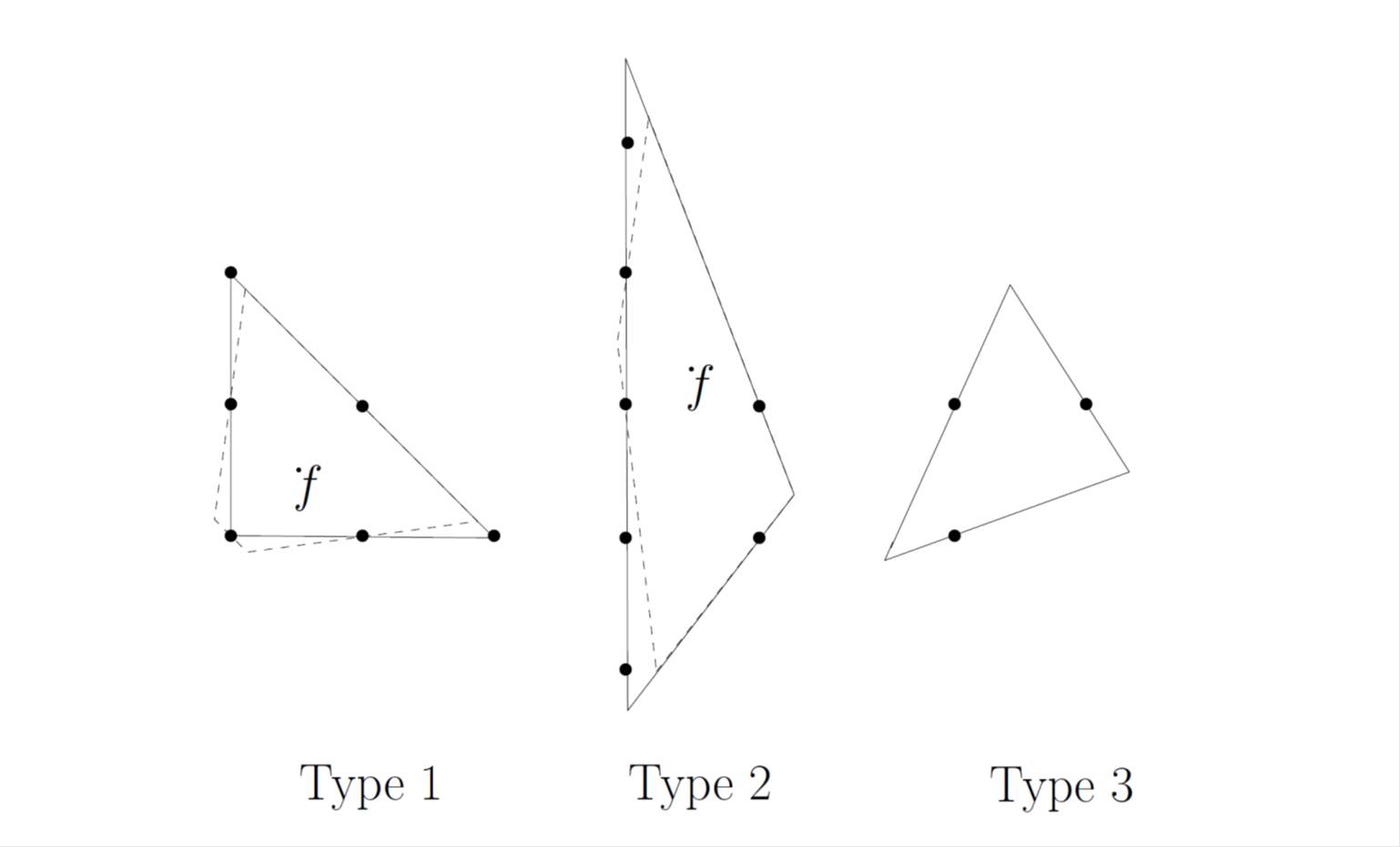}
\caption{Type-3 maximal lattice-free triangles cannot be approximated by maximal lattice-free quadrilaterals (Basu \emph{et al.}~\cite{BasuBCM11}).}
\label{Apprx_T_C}
\end{center}
\end{figure}

\section{Our Contribution}
\paragraph{}
In Section 3, we address the problem of determining how well the quadrilateral closure approximates the RCP. This amounts to finding the smallest possible value $\alpha_Q$ of the parameter $\alpha \geq 1$ such that $ R_f (r_1,\dots, r_k) \subseteq Q_f (r_1,\dots, r_k) \subseteq \alpha R_f (r_1,\dots, r_k)$. To the best of our knowledge, the tightest upper bound - published or unpublished, of 2, appears in Basu \emph{et al.}~\cite{BasuBCM11} i.e. we know that $\alpha_Q \leq 2$. We improve this bound from 2 to 1.71. 

In Section 4, we take up the problem of approximation of the RCP by the triangle closure. This means finding the smallest possible value $\alpha_T$ of the parameter $\alpha \geq 1$ such that $ R_f (r_1,\dots, r_k) \subseteq T_f (r_1,\dots, r_k) \subseteq \alpha R_f (r_1,\dots, r_k)$. To the best of our knowledge, the tightest lower bound - published or unpublished, of 1.125, appears in Cornu\'{e}jols \emph{et al.}~\cite{LB10} i.e. we know that $\alpha_T \geq 1.125$. In Section 4, we generalize this lower bound of 1.125 to an infinite subclass of quadrilaterals.

\section{Upper bound for  approximation of the RCP by the quadrilateral closure}
\paragraph{}
In this section, we address the problem of finding the tightest possible upper bound $\alpha_Q$ for approximation of the RCP by the quadrilateral closure.

\subsection{Existing state-of-the-art bound of 2}
\paragraph{}
The current tightest upper bound, which is from Basu \emph{et al.}~\cite{BasuBCM11} is 2, meaning $\alpha_Q \leq 2$.  We summarize here the theoretical tools which were used and which we will borrow from this work as well. We also introduce some new notation, which is more compact, for our analysis. An elegant tool used in the analysis is Goemans theorem.
\begin{theorem}
(Goemans~\cite{Goemans95}): Let $A \subseteq \mathbb{R}^n_+ \setminus \{0\}$ be a polyhedron of the form $ A = \{x: a^ix \geq b_i \hspace{4pt} \forall i = 1,\dots ,m\}$ where $ a^i \geq 0$ and
$b^i \geq 0 \hspace{4pt} \forall i = 1,\dots, m$. Let $\alpha > 0$ be a scalar. Define $\alpha A = \{x: \alpha a^ix \geq b^i \hspace{4pt} \forall
i = 1,\dots,m \}.$ If the convex set $B$ is a relaxation of $A$, then the smallest $\alpha \geq 1$ such that $B \subseteq \alpha A$ is given by
\begin{equation}
 \max_{i = 1,\dots, m} \{ { {   \frac { b_i}  { \inf \{\large a^i x: x \in B \} }} : \large b_i > 0 \}}. \nonumber
\end{equation}
\label{goemans}
\end{theorem}
The above theorem implies that we need to optimize along the direction of each facet-defining inequality of the polyhedron to be approximated while remaining within the convex set used for approximation and take the best amongst such optimals.
To determine how well the quadrilateral closure approximates the RCP, we need to obtain 
\begin{equation}
\inf \limits_{\text{(}f \text{, MLF set }B \text{): } \psi_B^f  \text{ defines a facet of } R_f}  \inf \limits_{\large s} \{ \sum_{i = 1}^k \psi_B^f (r_i) s_i : s \in Q_f(r_1,\dots,r_k) \}.
\end{equation}
 If a facet defining inequality arises from a maximal lattice-free quadrilateral, then its contribution to the maximization term in Theorem \ref{goemans}  cannot exceed 1. Hence, only facet-defining inequalities from maximal lattice-free triangles need to be considered. i.e.

\begin{align}
\inf \limits_{\text{(}f \text{, MLF set }B \text{): } \psi_B^f  \text{ defines a facet of } R_f}   \text{            }
\inf \limits_{s} \{ \sum_{i = 1}^k \psi_B^f (r_i) s_i : s \in Q_f(r_1,\dots,r_k) \} \nonumber \\
=
\inf \limits_{\text{(}f \text{, MLF triangle }T \text{): } \psi_T^f  \text{ defines a facet of } R_f}  
\text{            } \inf  \limits_{s} \{ \sum_{i = 1}^k \psi _T^f(r_i) s_i : s \in Q_f(r_1,\dots,r_k) \}.
\end{align}

However, we know that (see Figure \ref{Apprx_T_C})  Type-2 triangles can be approximated arbitrarily closely by maximal lattice-free quadrilaterals (Basu \emph{et al.}~\cite{BasuBCM11}). Also, Type-1 triangles can be approximated arbitrarily closely by Type-2 triangles and thus by maximal lattice-free quadrilaterals. Hence, we need only consider Type-3 triangles i.e.
\begin{align}
\inf \limits_{\text{(}f \text{, MLF triangle }T \text{): } \psi_T^f  \text{ defines a facet of } R_f}  
\text{            } \inf  \limits_{s} \{ \sum_{i = 1}^k \psi _T^f(r_i) s_i : s \in Q_f(r_1,\dots,r_k) \}
\\ =
 \inf \limits_{\text{(}f \text{, Type-3 MLF triangle }T \text{): } \psi_T^f  \text{ defines a facet of } R_f} \inf \limits_{s} \{ \sum_{i = 1}^k \psi_T^f (r_i) s_i : s \in Q_f(r_1,\dots,r_k) \}.
\end{align}

Basu \emph{et al.}~\cite{BasuBCM11} use the positive homogeneity of $\psi$ and that $\psi$ is unity for all points on the boundary of T (with shift of origin to $f$). When $\psi_B(r_j)  > 0 $, we define $r_j$ such that that $f+r_j$ is on the boundary of T. This can  be done with the requisite scaling, as it does not  change the optimal value of $ \inf \limits_{s} \{ \sum_{i = 1}^k \psi_T^f (r_i) s_i : s \in Q_f(r_1,\dots,r_k) \}$. Thus from now on, we analyze the LP  $ \inf \limits_{s} \{ \sum_{i = 1}^k s_i : s \in Q_f(r_1,\dots,r_k) \}$.

The following theorem indicates that we only need to consider the corner rays in our analysis.
\begin{theorem}
(Basu \emph{et al.}~\cite{BasuBCM11}): Let $B_1,\dots,B_m$ be lattice-free convex sets with $f$ in the interior of $B_p$, $ p = 1,\dots,m.$ Let $R_c 
\subseteq \{1,\dots,k\}$ be a subset of the ray indices such that for every ray $r_j$ with $j \notin R_c, r^j$ is the convex combination of some two rays in $R_c$. Then
 \center
$ 
(\min$ $  \sum_{i=1}^k s_i \hspace{4pt}: \hspace{4pt} \sum_{i = 1}^k \psi_{B_p} (r_i) s_i \geq 1 \hspace{4pt} 
\forall p = 1,\dots m \hspace{4pt}, s \geq 0)$  
\center
$= (\min$ $  \sum_{i \in R_c} s_i \hspace{4pt}: \hspace{4pt} \sum_{i \in R_c} \psi_{B_p} (r_i) s_i \geq 1 \hspace{4pt} 
\forall p = 1,\dots m \hspace{4pt}, s \geq 0)$.
\end{theorem}
Thus  if $r_1, \dots, r_3$ are the corner rays of a maximal lattice-free triangle $T$, then 
\begin{equation}
\inf \limits_{s} \{ \sum_{i = 1}^k s_i : s \in Q_f(r_1,\dots,r_k) \} = 
\inf \limits_{s} \{ \sum_{i = 1}^3 s_i : s \in Q_f(r_1,\dots,r_k) \}.
\end{equation} 
 
To prove $\alpha_Q \leq 2$, Basu \emph{et al.}~\cite{BasuBCM11} consider a relaxation, say $Rel_1(Q_f(r_1,\dots,r_k))$ of $Q_f(r_1,\dots,r_k)$, using just two Type-2 triangles - namely $T_1$ and $T_2$ (see Figure \ref{Basu_QC_UB}).   An affine transformation enables us to consider a tilted co-ordinate system, as shown in Figure \ref{Basu_QC_UB}. The Type-3 triangle (say $T$) is defined by the vertices $f+r_1$, $f+r_2$ and $f+r_3$. $T_1$ is given by the side of the Type-3 triangle passing through (0,0), the side of the Type-3 triangle passing through (0,1) and the line $x_1 = 1$. $T_2$ is given by the side of the Type-3 triangle passing through (0,0), the side of the Type-3 triangle passing through (1,0) and the line $x_2 = 1$.  Basu \emph{et al.}~\cite{BasuBCM11} prove the following theorem for the relaxation. 
\begin{theorem}
Let $T$ be a triangle of Type-3 with corresponding minimal function $\psi$ and generating a facet $\sum_{i = 1}^k \psi _T(r_i) s_i \geq 1$ of $R_f(r_1,\dots,r_k)$. Then
$\inf \{ \sum_{i = 1}^k \psi_T (r_i) s_i : s \in Rel_1(Q_f(r_1,\dots,r_k) ) \} \geq \frac{1}{2}.$ 
\end{theorem}

\begin{figure}
\begin{center}
\includegraphics[height=15cm]{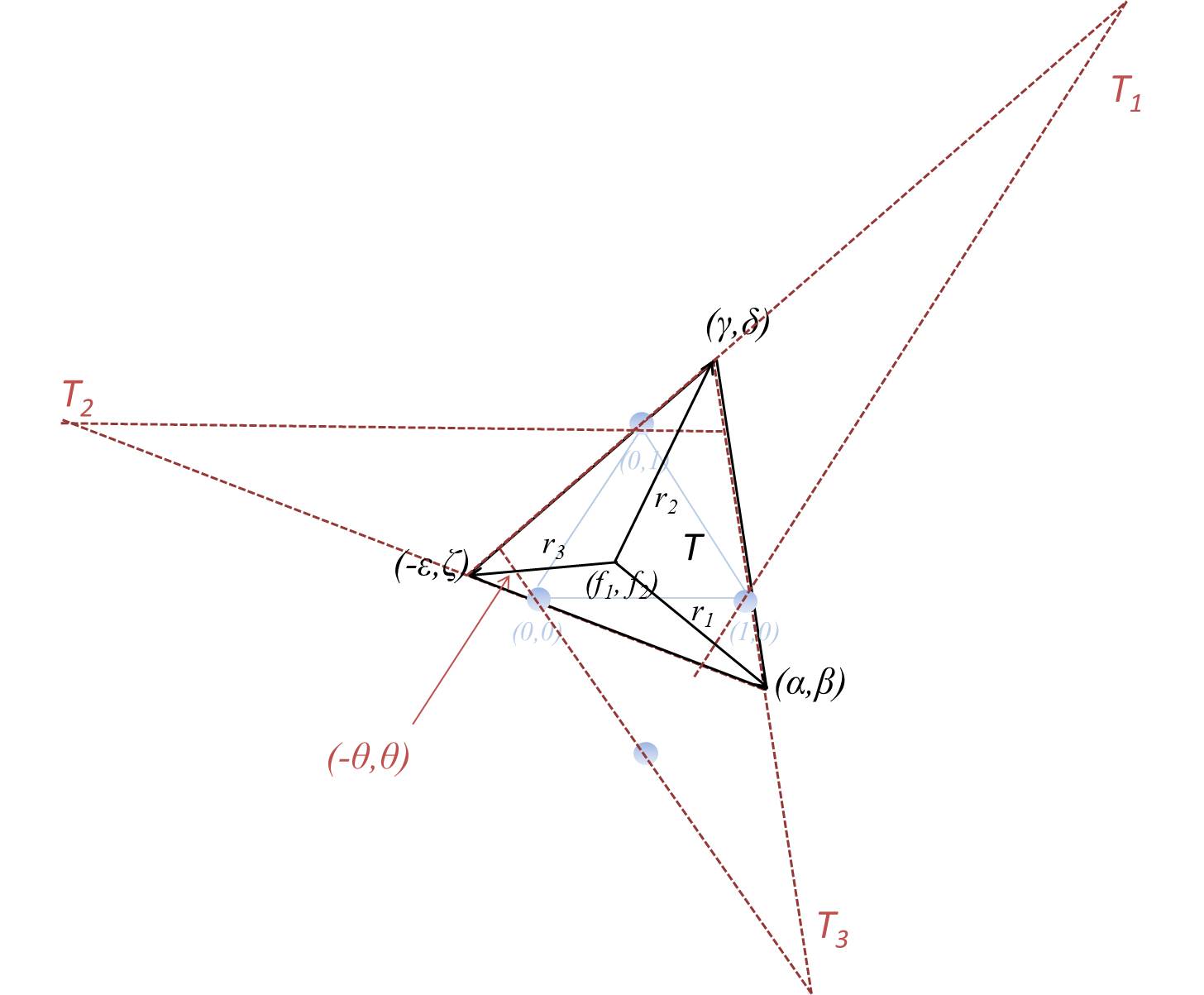}
\caption{Using three instead of two triangles for the relaxation of the quadrilateral closure gives us an improved bound.}
\label{Basu_QC_UB}
\end{center}
\end{figure}

An important result  used in the analysis is the following.
\begin{theorem}
(Basu \emph{et al.}~\cite{BasuBCM11})
Let $\mathcal{F}$ be the family of triangles formed by three lines given by\\
\begin{equation}
 Line \text{ 1 }:  -\frac{x_1}{t_1} + x_2 = 1\text{  with } 0<t_1<\infty; \nonumber 
\end{equation}
\begin{equation}
Line  \text{ 2 }:  t_2x_1 + x_2 = 0 \text{  with } 0<t_2<1; \nonumber
\end{equation}
\begin{equation}
Line  \text{ 3 }:  x_1 + \frac{x_2}{t_3} = 1 \text{  with } 1<t_3<\infty.
\end{equation}
Any Type-3 triangle is either a triangle from $\mathcal{F}$ or the reflection of a triangle from $\mathcal{F}$ about the line $x_1 = x_2$.
\label{Bsau_refl}
\end{theorem}

Thus, by reflexive symmetry, it suffices to consider the family of triangles $\mathcal{F}$.\\

Using rotational symmetry, Basu \emph{et al.}~\cite{BasuBCM11} break the analysis into two cases:\\
Case (i): $f_1\leq \frac{1}{2}$, $f_2 \leq \frac{1}{2}$. \\
Case (ii): $f_1<0$, $f_1 + f_2 \leq \frac{1}{2}$. \\
 For any Type-3 triangle $T$, they show that $\inf \{ \sum_{i = 1}^k \psi_T (r_i) s_i : s \in Rel_1(Q_f(r_1,\dots,r_k) ) \} \geq \frac{1}{2}$ for Case (i) and $\inf \{ \sum_{i = 1}^k \psi_T (r_i) s_i : s \in Rel_1(Q_f(r_1,\dots,r_k) ) \} \geq 0.586$ for Case (ii). This gives an overall minimum of 0.5 and thus a bound of 2. The interested reader may refer to Basu \emph{et al.}~\cite{BasuBCM11} for the details.

\subsection{Our Improved Bound of 1.71}
\paragraph{}
We improve the existing bound using an additional triangle, namely $T_3$ (see Figure \ref{Basu_QC_UB}) instead of just $T_1, T_2$ in the LP relaxation. $T_3$ is given by the side of the Type-3 triangle passing through (0,1), the side of the Type-3 triangle passing through (1,0) and the line $x_1 + x_2 = 0$.  Let the relaxation obtained be  $Rel_2(Q_f(r_1,\dots,r_k))$.
We prove the following theorem which gives us our improved bound.
\begin{theorem}
\begin{equation}
 \inf \limits_{\text{(}f \text{, Type-3 MLF triangle }T \text{): } \psi_T^f  \text{ defines a facet of } R_f} \inf \limits_{s} \{ \sum_{i = 1}^3 \psi_T^f (r_i) s_i : s \in Q_f(r_1,\dots,r_k) \} \geq 
0.586
\end{equation}
 \label{QC_YPA_new}
\end{theorem}
\begin{corollary}
$ Q_f (r_1,\dots, r_k) \subseteq 1.71 R_f (r_1,\dots, r_k)$.
\end{corollary}
\paragraph{}
\emph{Proof of Theorem \ref{QC_YPA_new}}

We divide the proof into two cases, depending on the position of the fractional solution (see Figure \ref{pic_QC_UB_2}) .\\
Case 1: The region (bounded in solid green in Figure \ref{pic_QC_UB_2})  in the Type-3 triangle common to all the three Type-2 triangles, namely $T_1$, $T_2$ and $T_3$. \\
Case 2:  the region in the Type-3 triangle lying in exactly two of the three Type-2 triangles. \\

Case 2: The following result follows easily.
\begin{lemma}
Let $f$ belong to  the region in a Type-3 triangle lying in exactly two of the three Type-2 triangles used for approximation. Then \\
\begin{equation}
 \inf \limits_{\text{(}f \text{, Type-3 MLF triangle }T \text{): } \psi_T^f  \text{ defines a facet of } R_f} \inf \limits_{s} \{ \sum_{i = 1}^3 \psi_T^f (r_i) s_i : s \in Q_f(r_1,\dots,r_k) \} \geq 0.586. \nonumber
\end{equation}
\label{QC_UB_Case2}
\end{lemma}

\begin{proof}
The region defined by our Case 2 is contained in the region defined by Case (ii) in Basu \emph{et al.}~\cite{BasuBCM11} for which a minimum of 0.586 is known. The result follows.
\end{proof}

Case 1:
We now prove a result for Case 1, which gives us the improved bound.
\begin{lemma}
Let $f$ belong to  the region in a Type-3 triangle lying in all three of the Type-2 triangles used for approximation. Then \\
\begin{equation}
\inf \limits_{\text{(}f \text{, Type-3 MLF triangle }T \text{): } \psi_T^f  \text{ defines a facet of } R_f} \inf \limits_{s} \{ \sum_{i = 1}^3 \psi_T^f (r_i) s_i : s \in Rel_2(Q_f(r_1,\dots,r_k)) \} \geq 
0.633. \nonumber
\end{equation}
\label{QC_UB_L1}
\end{lemma}
\begin{proof}

The minimization problem $ \inf \limits_{s} \{ \sum_{i = 1}^3 \psi_T^f (r_i) s_i : s \in Rel_2(Q_f(r_1,\dots,r_k)) \}$  is the same as the LP
\begin{align}
\min \hspace{4pt}  \sum_{i = 1}^3 s_i  \hspace{4pt} s.t. \nonumber\\ 
  \frac{\alpha - f_1}{1- f_1} s_1+ s_2+s_3  \geq 1, \nonumber \\ 
 s_1 + \frac{\delta - f_2}{1- f_2} s_2 + s_3  \geq 1 , \\ 
 s_1 + s_2 + \frac{\epsilon+f_1}{\theta + f_1} s_3 \geq 1, \nonumber\\
s \in \mathbb{R}^3_+ \nonumber.
\end{align}


The dual of this LP is given by
\begin{align}
\max \hspace{4pt}   \sum_{i = 1}^3 w_i  \hspace{4pt} s.t.  \nonumber \\
\frac{\alpha - f_1}{1- f_1}w_1 + w_2 + w_3  \leq 1, \\
 w_1 + \frac{\delta - f_2}{1- f_2}w_2 + w_3  \leq 1,   \nonumber\\
 w_1 + w_2 + \frac{\epsilon+f_1}{\theta + f_1}w_3  \leq 1  \nonumber, \\ 
w \in \mathbb{R}^3_+  \nonumber.
\end{align}
\vspace {4pt}
The dual optimal solution is given by solving 
\begin{align}
 \frac{\alpha - f_1}{1- f_1}w_1 + w_2 + w_3  = 1, \\
 w_1 + \frac{\delta - f_2}{1- f_2}w_2 + w_3 =  1,   \nonumber\\
 w_1 + w_2 + \frac{\epsilon+f_1}{\theta + f_1}w_3  =  1  \nonumber.
\end{align} 
Let
\begin{equation}
1 + a = \frac{\alpha - f_1}{1- f_1}, 1+b = \frac{\delta - f_2}{1- f_2}, 1+c =  \frac{\epsilon+f_1}{\theta + f_1}.
\end{equation}
\\where  $ a >0, b>0, c>0, d>0$. Then this system of equations is of the form
\begin{align}
(1+a)w_1 + w_2 + w_3  = 1,  \nonumber \\
 w_1 + (1+b)w_2 + w_3  = 1 ,   \\
 w_1 + w_2 + (1+c)w_3  = 1   \nonumber.
\end{align} 
Solving this system of equations gives the optimal objective function value as
\begin{equation}
w_1 + w_2 + w_3 =
 1 - \Bigg( \dfrac{1}{1 +\dfrac{1}{a} + \dfrac{1}{b} + \dfrac{1}{c}}\Bigg).
\end{equation}
We have  that
\begin{equation}
\dfrac{1}{a} + \dfrac{1}{b} + \dfrac{1}{c}  = 
\dfrac{1-f_1}{\alpha - 1}+ \dfrac{1-f_2}{\delta - 1} +  \dfrac{\theta + f_1}{ \epsilon-\theta}.
\end{equation}\\
Firstly, it is  worth noting that the denominator of each term is simply the distance of  a corner point of the Type-3 triangle from the point of intersection of the side  of the corresponding Type-2 triangle (which cuts off a part of the Type 3 triangle including the corner point) and the corner ray from the fractional solution to the corner point. The numerator is the distance of the  fractional solution $f$ from this intersection point. Let the slope of the line joining $(\alpha,\beta)$ and $(\gamma, \delta)$ be $-u$, line joining $(\gamma, \delta)$ and $(-\epsilon, \zeta)$ be $w$, line joining $(-\epsilon, \zeta)$ and $(\alpha,\beta)$ be $-v$. From Theorem \ref{Bsau_refl} , it follows that  $u>1$, $,0<v<1$ and $w>0$.
\begin{lemma}
\begin{equation}
\dfrac{1-f_1}{\alpha - 1}+ \dfrac{1-f_2}{\delta - 1} +  \dfrac{\theta + f_1}{ \epsilon-\theta}
 =  \\ (1-f_1) (\frac{u}{v} - 1) + (1-f_2) \dfrac{u+w}{w(u-1)} + \dfrac{(f_1+f_2)(v+w)}{1-v}.
\end{equation}
\label{app_proof_1}
\end{lemma}
\begin{proof}
See Appendix A.1.
\end{proof}

\begin{figure}
\begin{center}
\includegraphics[height=15cm]{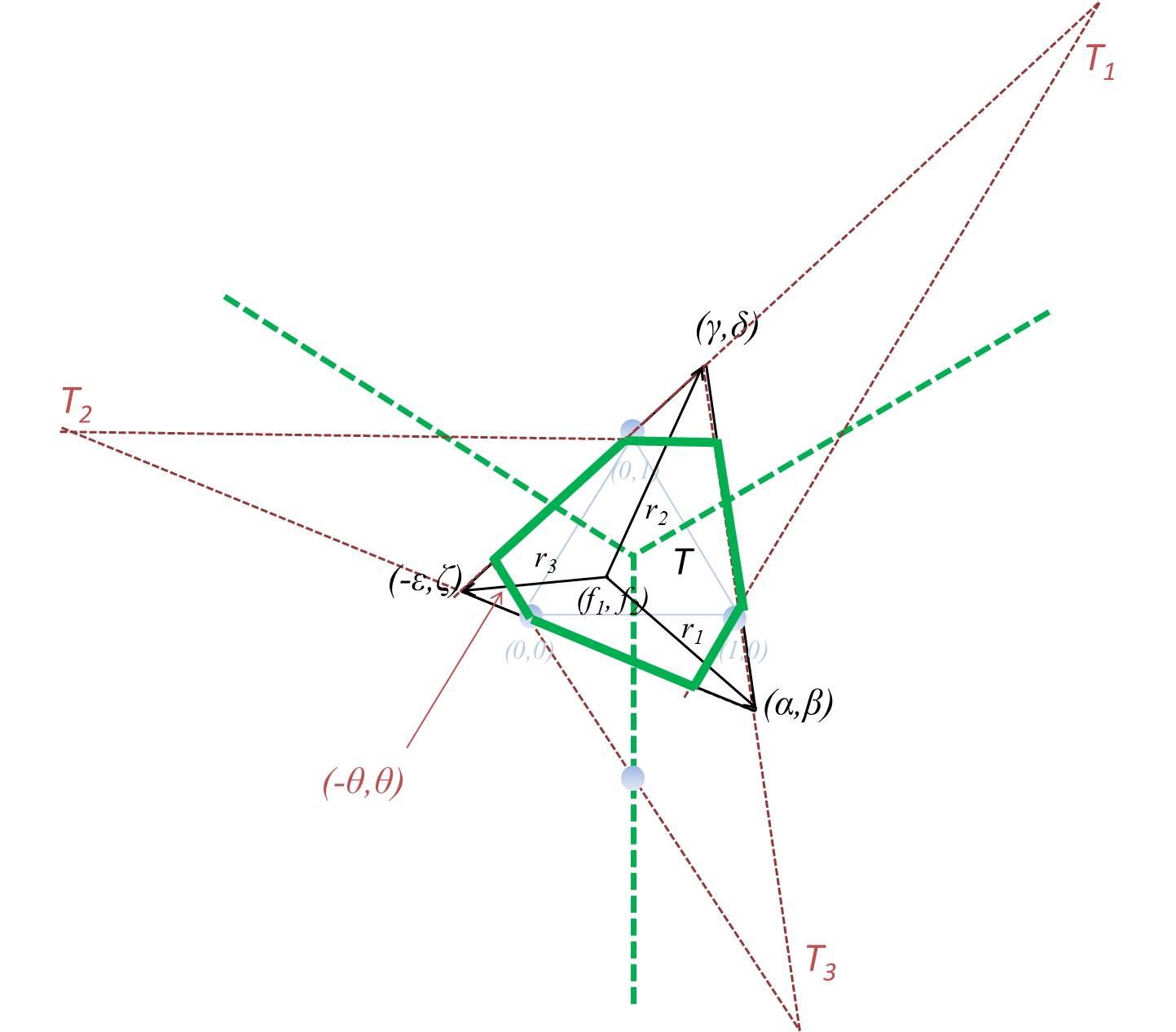}
\caption{Breakup of cases in our analysis.}
\label{pic_QC_UB_2}
\end{center}
\end{figure}

We prove the following result which gives us Theorem  \ref{QC_YPA_new}  and hence the improved bound.
\begin{lemma}
Let $f$ belong to  the region in a Type-3 triangle lying in all three of the Type-2 triangles used for approximation.  Then \\
\begin{equation}
 \dfrac{1-f_1}{\alpha - 1}+ \dfrac{1-f_2}{\delta - 1} +  \dfrac{\theta + f_1}{ \epsilon-\theta} \geq \sqrt{3}. \nonumber
\end{equation}
\label{FA_QC_UB}
\end{lemma}
\begin{proof}
Case 1 is given by the region bounded by six lines i.e.
i.e. $x_2  \leq  w x_1 + 1$, $x_2 \leq  1$, $x_2  \leq   -u(x_1 - 1)$, $x_1 \leq  1$, $x_2 \geq  -vx_1$ and $x_2 \geq -x_1$.
By rotational symmetry (see the dotted green lines in Figure \ref{pic_QC_UB_2}), it suffices to examine the region, say $Q$, given by 
$x_1 \leq \dfrac{1}{2}$, $x_2 \geq -vx_1$, $x_2 \geq -x_1$, $x_2  \leq  w x_1 + 1$ and $x_2  \leq  -\dfrac{1}{2}(x_1 - 1)$. 
This region $Q$ after rotations by 120 degrees and 240 degrees covers the whole Type-3 triangle.

\end{proof}
This  gives us Lemma \ref{QC_UB_L1}.
\end{proof}
Lemmas \ref{QC_UB_Case2} and \ref{QC_UB_L1} give us Theorem  \ref{QC_YPA_new}. \qed

\section{Generalization of the lower bound for the approximation of RCP by the triangle closure}
\paragraph{}
In this section, we address the problem of finding a lower bound on the approximation parameter $\alpha_T$ for the approximation of the RCP by the triangle closure. To achieve this, we attempt to find a point in the triangle closure which requires the largest scaling for a valid inequality (of the RCP) defined by a maximal lattice-free quadrilateral such that the point then satisfies the scaled inequality. The best bound known to us in literature - published or unpublished, is due to Cornu{\'e}jols \emph{et al.}~\cite{LB10} and  is 1.125, meaning $\alpha_T \geq 1.125$. The maximal lattice-free quadrilateral (say Q) chosen by them has the vertices $v^1 = f+r_1 = (1.4,0.8)$, $ v^2 = f+r_2 =  (0.8,-0.4)$, $v^3 = f+r_3 = (-0.4, 0,2)$, $v^4 =f+r_4 =  (0.2, 1.4)$ with the integral points (0,0), (0,1), (1,0) and (1,1) lying on the boundary of Q and $f = (0.5, 0,5)$. Q defines the valid inequality $\sum_{j=1}^{4} s_j \geq 1$ for $R_f(r_1,\dots, r_k) $. The point chosen by them is $\bar{s}= (\frac{2}{9}, \frac{2}{9}, \frac{2}{9}, \frac{2}{9})$ which can be proved to lie in the triangle closure but which requires a scaling of 1.125 of  the valid inequality defined by Q in order that $ \bar{s}$ satisfies the transformed inequality. The interested reader may refer to Cornu{\'e}jols \emph{et al.}~\cite{LB10} for further details.

Here, we generalize the result to prove that 1.125 is the best bound obtainable from an infinite subclass of maximal lattice-free quadrilaterals which are symmetric w.r.t. the four lattice points - namely  (0,0), (0,1), (1,0) and (1,1), have these four points on their boundary and have $f = (0.5, 0.5)$. More formally, we prove the following theorem.

\begin{theorem}

Let S be the set of minimal valid inequalities defined by the subclass (see Figure \ref{TC_LB_1}) of maximal lattice-free quadrilaterals (say G), with $f = (0.5, 05)$, having vertices given by $f+r_1 =  (1+a, 1-b), f+r_2 = (1-b, -a), f+r_3 = (-a, b), f+r_4 = (b, 1+a)$ for some $a, b >0 $ with the relation $a^2 = b(1-b)$. Let $ X = T_f (r_1,\dots, r_k) \cap S$.
Then $\min (\alpha >0:   T_f (r_1,\dots, r_k)  \subseteq \alpha X)$ = 1.125.
\end{theorem}

\begin{figure}
\begin{center}
 \includegraphics[height=9cm]{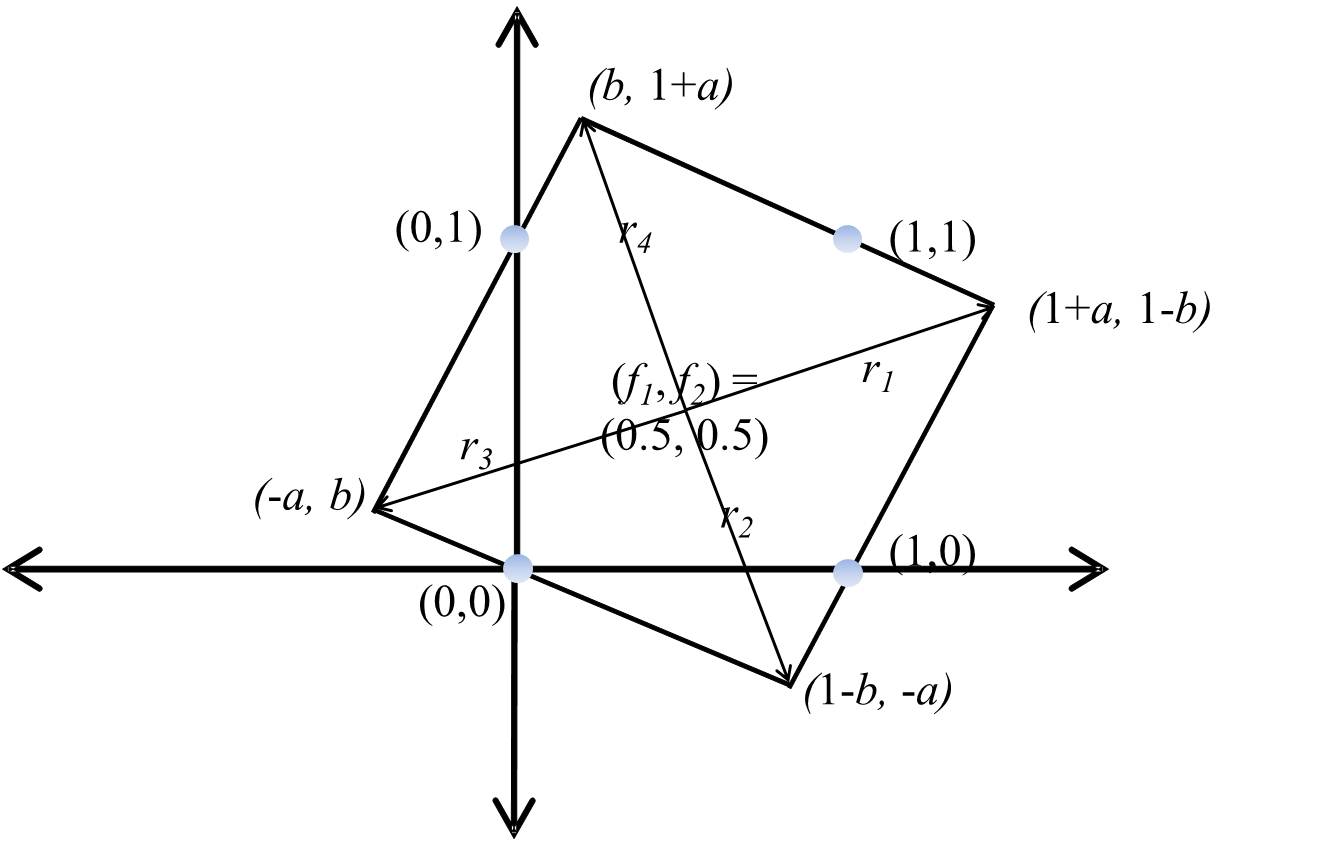}
\caption{A general quadrilateral of our subclass.}
\label{TC_LB_1}
\end{center}
\end{figure}
\begin{proof}Let $Q$ be any maximal lattice-free quadrilateral in G.  Q defines the valid inequality 
\begin{equation}
\sum_{j=1}^{4} s_j \geq 1.
\end{equation}
 for $R_f(r_1,\dots, r_k) $. Let B be any maximal lattice-free triangle.
Amongst the edges of B, atleast one edge has two of the points amongst (0,0), (0,1),(1,0) and (1,1) on one side and $f$ on the opposite side. Without loss of generality, this edge has (1,0) and (1,1) on one side and $f$ on the other side (Cornu{\'e}jols \emph{et al.}~\cite{LB10}). This implies 
\begin{equation}
\psi_B(r_1) \geq \frac{0.5+a}{0.5} = \text{threshold}_1 \hspace{4pt}.
\end{equation} 
Treating $(r_1, r_4)$ as the basis, let $(\lambda_1, \lambda_2)$ be the co-ordinates of (1,1). Then it is easily checked that $(\lambda_1, \lambda_2)$ are also the co-ordinates of (1,0) with respect to the basis $(r_2, r_1)$,  (0,0) with respect to the basis $(r_3, r_2)$ and (0,1) with respect to the basis $(r_4, r_3)$. It is also easily checked that \\
\begin{equation}
\lambda_1 = \frac{0.5(1+a-b)}{(0.5+a)^2+(b-0.5)^2}, \hspace{4pt}
\lambda_2 = \frac{0.5(a+b)}{(0.5+a)^2+(b-0.5)^2}.
\end{equation}

Let $f+r_2$ intersect $x_1 = 1$ at P (see Figure \ref{TC_LB_2}). Let the line joining P with origin intersect $f+r_3$ at $P'$. It can be easily checked that the x co-ordinate of $P'$ is given by 
\begin{equation} 
\frac{(a+b)(2b-1)}{(2b-1)^2+(2a+1)(a+b)}.
\end{equation}
Let
\begin{equation}
\text{threshold}_2 = \dfrac{0.5 + a}{0.5 - \dfrac{(a+b)(2b-1)}{((2b-1)^2+(2a+1)(a+b)}}.
\end{equation}
Let
 \begin{equation} 
m = \frac{\lambda_1}{( 2 + (\lambda_1 - \lambda_2)(\text{threshold}_1 + \text{threshold}_2) )}.
\end{equation}
\begin{figure}
\begin{center}
 \includegraphics[height=9cm]{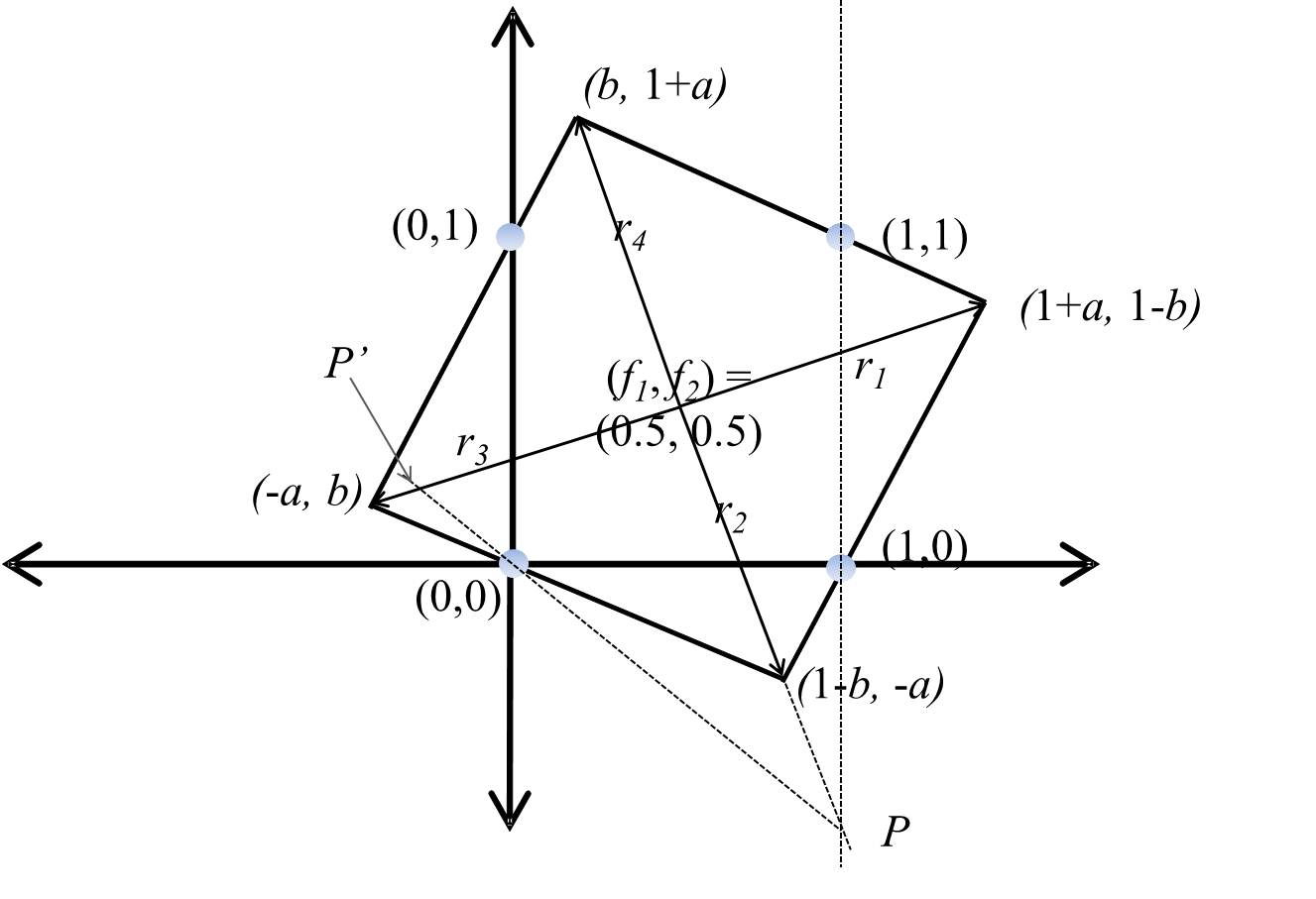}
\caption{A construction which determines the cases for our proof.}
\label{TC_LB_2}
\end{center}
\end{figure}
Let the point $\bar{s}$ be defined as 
\begin{equation}
\bar{s} = (m,m,m,m).
\end{equation}
We now follow two cases for our analysis. Case 1 assumes $\psi_B(r_3) \geq \text{threshold}_2$ and Case 2 assumes $\psi_B(r_3) < \text{threshold}_2$.

Case 1: We note that (1,0) has to be above the line $\psi_B(r_1)z_1 + \psi_B(r_2)z_2 = 1$, where $(z_1, z_2)$ denote the co-ordinates of (1,0) with respect to the basis $(r_1,r_2)$. This implies
\begin{equation}
\lambda_2  \psi_B(r_1) + \lambda_1 \psi_B(r_2) \geq 1 .
\label{eq:ref1}
\end{equation}
Similarly,
\begin{equation}
\lambda_2 \psi_B(r_3) + \lambda_1  \psi_B(r_4) \geq 1.
\label{eq:ref2}
\end{equation}
From equations (\ref{eq:ref1}) and  (\ref{eq:ref2}) and noting that $\lambda_1 \geq \lambda_2$, we have 
\begin{equation}
\lambda_1  (\psi_B(r_1) + \psi_B(r_2) + \psi_B(r_3)+ \psi_B(r_4)) \geq 2 + (\lambda_1 - \lambda_2)( \psi_B(r_1) +  \psi_B(r_3)) .
\end{equation}
Hence
\begin{equation}
\lambda_1  (\psi_B(r_1) + \psi_B(r_2) + \psi_B(r_3)+ \psi_B(r_4)) \geq 2 + (\lambda_1 - \lambda_2)(\text{threshold}_1 + \text{threshold}_2) .
\end{equation}
Hence
\begin{equation}
\frac{\lambda_1}{( 2 + (\lambda_1 - \lambda_2)(\text{threshold}_1 + \text{threshold}_2 ))}(\psi_B(r_1) + \psi_B(r_2) + \psi_B(r_3)+ \psi_B(r_4)) \geq 1.
\label{eq:ref3}
\end{equation}

From equation (\ref{eq:ref3}), $\bar{s}$ satisfies the minimal inequality defined by the maximal lattice-free triangle $B$. Hence $\bar{s}$  lies in the triangle closure. However, to satisfy the quadrilateral inequality $  \sum_{j=1}^{4} s_j \geq 1$ defined by Q, we need to scale $\bar{s}$ by 
\begin{equation}
q =  \frac{( 2 + (\lambda_1 - \lambda_2)(\text{threshold}_1 + \text{threshold}_2 )}{4\lambda_1}
\end{equation}
\begin{equation}
= \big((2\sqrt{(b(1-b))}+1)^2 + (2b-1)^2\big)\frac{(\sqrt{(b(1-b))}-3b+2)}{4(\sqrt{b(1-b)} - b+1)^2}.
\label{eq:refbound}
\end{equation}

Case 2: As in Cornu{\'e}jols \emph{et al.}~\cite{LB10}, without loss of generality, we can assume that $w$ is a vertex of $B$ and that the line separating (1,0), (1,1) from $f$ is $x_1 = 1$. Hence, 
\begin{equation}
\psi_B(r_1) = \frac{0.5+a}{0.5} .
\label{eq:ref4}
\end{equation}
The maximality of B implies that each of (0,0) and (1,1) lies on a side of B. Hence
\begin{equation}
\lambda_1  \psi_B(r^3) + \lambda_2  \psi_B(r^2)  = 1.
\label{eq:ref5}
\end{equation}
\begin{equation}
 \lambda_1  \psi_B(r^4) + \lambda_2  \psi_B(r^3) = 1.
\label{eq:ref6}
\end{equation}
It can easily be checked that equations (\ref{eq:ref4}), (\ref{eq:ref5}), (\ref{eq:ref6}) imply equation (\ref{eq:ref3}). Hence, in this case as well, $\bar{s}$  lies in the triangle closure and needs to be scaled by $q$ in order to satisfy the quadrilateral inequality $  \sum_{j=1}^{4} s_j \geq 1$ defined by Q.

From Cases 1 and 2, we obtain the bound $q$ given by equation (\ref{eq:refbound}). It is not difficult to confirm that the global maximum for this function of $b$  is 1.125 at $b = 0.2$. This finishes the proof.
\end{proof}

The above result implies that for the structure $X$ defined by the intersection of the triangle closure and the set $S$ of minimal valid inequalities from maximal lattice-free quadrilaterals in the above subclass $G$, the  approximation factor for approximation by the triangle closure is known to be exactly 1.125.

\section{Conclusions}
\paragraph{}
In this paper, for the two dimensional case, we improve on the state-of-the-art upper bound for the approximation of the Relaxed Corner Polyhedron (RCP)  purely by the quadrilateral closure, from 2 to 1.71. We  generalize the state-of-the-art lower bound of 1.125 for the approximation of the RCP purely by the triangle closure, to an infinite subclass of maximal lattice-free quadrilaterals.
%
%
%
%


\section{Acknowledgements}
\paragraph{}
The author would like to thank G\'{e}rard Cornu\'{e}jols for helpful discussions.

\bibliographystyle {abbrv}
\bibliography {bibtex_yogesh}

\section{Appendix A.1}

Here, we prove Lemma \ref{app_proof_1}.
\begin{proof}
\begin{equation}
x_2 - 0 = -u(x_1-1) \text{ gives } \beta = -u(\alpha -1).
\end{equation}
\begin{equation}
x_2= -vx_1 \text{ gives } \beta = -v\alpha.
\end{equation}
\begin{equation}
\text{Hence } \alpha (u - v) = u \text{ i.e. } \alpha = \frac{u}{u-v} \text{ i.e. } \alpha - 1= \frac{v}{u-v} \text{ i.e. }\frac{1}{\alpha-1} = \frac{u-v}{v} = \frac{u}{v} - 1. 
\end{equation}
\begin{equation}
x_2 -1 = w(x_1-0) \text{ gives } \delta - 1 = w\gamma. 
\end{equation}
\begin{equation}
x_2 - 0 = -u(x_1-1) \text{ gives } \delta = -u(\gamma -1). 
\end{equation}
\begin{equation}
\text{ Hence } \gamma =\frac{ \delta - 1 }{w} = -\frac{\delta}{u} +1 \text{ i.e. }
\delta = \frac{u(w+1)}{u+w} \text{ i.e. }
\delta - 1 = \frac{w(u-1)}{u+w} \text{ i.e. } \frac{1}{\delta - 1} = \frac{u+w}{w(u-1)}. 
\end{equation}
\begin{equation}
x_1 = -vx_1 \text{ gives } \zeta = v\epsilon. 
\end{equation}
\begin{equation}
x_2 -1 = w(x_1-0) \text{ gives } \zeta -1 = -w\epsilon. 
\end{equation}
\begin{equation}
\text{ Hence } \zeta = v\epsilon = -w\epsilon + 1 \text{ i.e. } \epsilon = \frac{1}{v +w} \text{ and } \zeta = \frac{v}{v +w}. 
\end{equation}
\begin{equation}
\text{ The line passing through } (f_1, f_2) \text{, } (-\epsilon, \zeta) \text{ is } 
(\theta - f_2) = \Bigg( \dfrac{f_2 - \dfrac{v}{ v+w}}{f_1 + \dfrac{1}{v+w}} \Bigg)(-\theta - f_1).
\end{equation}
\begin{equation}
\text{ Thus } \theta \Bigg( 1+ \dfrac{f_2 - \dfrac{v}{ v+w}}{f_1 + \dfrac{1}{v+w}}\Bigg) 
= f_2 - f_1 \Bigg( \dfrac{f_2 - \dfrac{v}{ v+w}}{f_1 + \dfrac{1}{v+w}} \Bigg). 
\end{equation}
\begin{equation}
\text{ Thus } \theta \Bigg( \dfrac{f_1 + \dfrac{1}{v+w}+ f_2 - \dfrac{v}{ v+w}}{f_1 + \dfrac{1}{v+w}}\Bigg) 
= \dfrac{f_2 \Big( f_1 + \dfrac{1}{v+w}\Big) - f_1 \Big( f_2 - \dfrac{v}{ v+w}\Big)}{f_1 + \dfrac{1}{v+w}}. 
\end{equation}
\begin{equation}
\text{ Thus } \theta \Bigg(\dfrac{(f_1 + f_2)(v+w) + (1-v)}{ v+w}\Bigg) 
= \dfrac{f_1v+f_2}{v+w}.
\end{equation}
\begin{equation}
\text{ Thus }
\theta = \dfrac{f_1 v + f_2}{(f_1+f_2)(v+w)+(1-v)}.
\end{equation}
\begin{equation}
\text{ Hence } 
\dfrac{ \theta + f_1}{\epsilon - \theta} = 
\dfrac{\dfrac{f_1 v + f_2}{(f_1+f_2)(v+w)+(1-v)} + f_1}
{\dfrac{1}{v+w} - \dfrac{f_1 v + f_2}{(f_1+f_2)(v+w)+(1-v)}}.
\end{equation}
\begin{equation}
= \dfrac{(v+w)\Bigg(f_1 v + f_2 + f_1\Big((f_1+f_2)(v+w)+(1-v)\Big)\Bigg)}
{(f_1+f_2)(v+w)+(1-v) - (f_1 v + f_2)(v+w) }.
\end{equation}
\begin{equation}
= \dfrac{(v+w)\Bigg(f_2 \Big(1+f_1(v+w)\Big) + f_1^2(v+w) + f_1 \Bigg)}
{f_1(v+w)(1-v)+(1-v) }. 
\end{equation}
\begin{equation}
\text{ i.e. } \dfrac{ \theta + f_1}{\epsilon - \theta} 
= \dfrac{(v+w)(f_1 +f_2) \Big(1+f_1(v+w)\Big) }.
{(1-v)\Big(1+f_1(v+w)\Big) }.
\end{equation}
\begin{equation}
\text{ Hence } 
\dfrac{ \theta + f_1}{\epsilon - \theta} \text = \dfrac{(f_1+f_2)(v+w)}{1-v}.
\end{equation}
\begin{equation}
\text{ Thus }
\dfrac{1-f_1}{\alpha - 1}+ \dfrac{1-f_2}{\delta - 1} + \dfrac{\theta + f_1}{ \epsilon-\theta}
= (1-f_1) (\frac{u}{v} - 1) + (1-f_2) \dfrac{u+w}{w(u-1)} + \dfrac{(f_1+f_2)(v+w)}{1-v}.
\end{equation}
\end{proof}

\end{document}